\documentclass[11pt]{article}
\usepackage{amssymb,amsmath,amsthm,sectsty,url,graphicx,fullpage,color}
\newsavebox{\theorembox}
\newsavebox{\lemmabox}
\newsavebox{\corollarybox}
\newsavebox{\propositionbox}
\newsavebox{\examplebox}
\newsavebox{\conjecturebox}
\newsavebox{\algbox}
\newsavebox{\qbox}
\newsavebox{\problembox}
\newsavebox{\definitionbox}
\newsavebox{\assumptionbox}
\newsavebox{\hypothesisbox}
\savebox{\theorembox}{\noindent\bf Theorem}
\savebox{\lemmabox}{\noindent\bf Lemma}
\savebox{\corollarybox}{\noindent\bf Corollary}
\savebox{\propositionbox}{\noindent\bf Proposition}
\savebox{\examplebox}{\noindent\bf Example}
\savebox{\conjecturebox}{\noindent\bf Conjecture}
\savebox{\algbox}{\noindent\bf Algorithm}
\savebox{\qbox}{\noindent\bf Question}
\savebox{\definitionbox}{\noindent\bf Definition}
\savebox{\problembox}{\noindent\bf Problem}
\savebox{\assumptionbox}{\noindent\bf Assumption}
\savebox{\hypothesisbox}{\noindent\bf Hypothesis}
\newtheorem{theorem}{\usebox{\theorembox}}
\newtheorem{lemma}[theorem]{\usebox{\lemmabox}}

\newtheorem{definition}{\usebox{\definitionbox}}

\def\bq{{\mathbf q}}
\def\bX{\mathbf{X}}

\def\bY{\mathbf{Y}}
\def\bx{\mathbf{x}}
\def\bu{\mathbf{u}}
\def\bv{\mathbf{v}}

\def\TV{{\rm TV}}
\def\bZ{\mathbf{Z}}
\def\<{\langle}
\def\bnd{\beta^{{\rm 2nd}}}
\def\lnd{\lambda^{{\rm 2nd}}}
\def\>{\rangle}
\def\bbS{\mathbb{S}}
\def\reals{\mathbb{R}}
\def\ok{{^{\otimes k}}}
\def\naturals{{\mathbb N}}
\def\E{{\mathbb E}} 

\def\prob{{\mathbb P}}
\def\cF{{\cal F}}

\def\de{{\rm d}}
\def\bone{{\mathbf{1}}}
\def\bG{\mathbf{G}}
\def\normal{{\sf N}}
\def\bA{\mathbf{A}}
\def\bR{\mathbf{R}}

\def\bB{\mathbf{B}}
\def\be{{\mathbf e}}
\def\ind{{\mathbb{I}}}
\def\eps{\varepsilon}
\def\ed{\stackrel{{\rm d}}{=}}

\def\bone{{\mathbf{1}}}
\def\sT{{\sf T}}

\def\trace{\text{\rm Tr}}

\begin{document}

\title{On the limitation of spectral methods:\\ From the Gaussian hidden clique
problem to rank one perturbations of Gaussian tensors}
\author{Andrea Montanari\thanks{Department of Electrical
    Engineering and Department of Statistics, Stanford
    University. {\tt montanari@stanford.edu}. Partially supported by the NSF grant CCF-1319979 and the grants AFOSR/DARPA
FA9550-12-1-0411 and FA9550-13-1-0036} \and Daniel Reichman\thanks{Computer Science department, Cornell University, Ithaca, NY, 14853. {\tt
    daniel.reichman@gmail.com}. Work done
at the Weizmann Institute and supported in part by The Israel Science Foundation (grant No. 621/12)}
    \and Ofer Zeitouni\thanks{Faculty of  Mathematics, Weizmann Institute, Rehovot 76100, Israel and Courant Institute, New York University.
            {\tt ofer.zeitouni@weizmann.ac.il}. Partially supported by a grant from
the Israel Science Foundation and the Herman  P. Taubman chair of Mathematics at the Weizmann Institute.}}
\date{November 18, 2014}
\maketitle
\begin{abstract}
We consider the following detection problem: given a realization of a
symmetric matrix $\bX$ of dimension $n$, distinguish between the hypothesis
that all upper triangular variables are i.i.d. Gaussians variables
with mean 0 and variance $1$ and the hypothesis where
$\bX$ is the sum of such matrix and an independent rank-one
perturbation.

This setup applies to the situation where
under the alternative,
there is a
planted principal submatrix $\bB$ of size $L$ for which all upper triangular
variables are i.i.d. Gaussians with mean $1$ and variance $1$, whereas
all other upper triangular elements of $\bX$ not in $\bB$ are i.i.d.
Gaussians variables with mean 0 and variance $1$. We refer to this as
the `Gaussian hidden clique problem.'

When $L=(1+\epsilon)\sqrt{n}$ ($\epsilon>0$), it is possible to solve this
detection problem with probability $1-o_n(1)$ by computing the
spectrum of $\bX$ and considering the largest eigenvalue of $\bX$.
We prove that this condition is tight in the following sense: when
$L<(1-\epsilon)\sqrt{n}$ no algorithm that examines only the
eigenvalues of $\bX$
can detect the existence of a hidden
Gaussian clique, with error probability vanishing as $n\to\infty$.

We prove this result as an immediate consequence of a more general
result on rank-one 
perturbations of $k$-dimensional Gaussian tensors.
In this context we establish a lower bound on the critical
signal-to-noise ratio below which a rank-one signal cannot be
detected.
\end{abstract}
\section{Introduction}

Consider the following detection problem.
One is given a symmetric matrix $\bX=\bX(n)$ of dimension $n$, such that
the ${n \choose 2}+n$ entries
$(\bX_{i,j})_{i\leq j}$
are \emph{mutually independent} random variables. Given (a realization of) $\bX$ one
would like to distinguish between the hypothesis
that all random variables $\bX_{i,j}$ have the same distribution
$F_0$ to the hypothesis where
there is a set $U \subseteq [n]$ so that all
random variables in the submatrix $\bX_{U}:=(\bX_{s,t}:s,t \in U)$
have a distribution $F_1$ which is different from the
distribution of all other elements in $\bX$ which are still
distributed as $F_0$.

 The same problem was recently studied in 
\cite{ADD,Desh} and, for the of the
asymmetric case (where no symmetry assumption is imposed on the independent
entries of $\bX$),  in
\cite{Balla,Kollar,Ma}. 
We refer to Section \ref{sec-relwork} for further discussion of the related literature.
An intriguing outcome of these works is that, while the two hypothesis
are statistically distinguishable as soon as $L\ge C\log n$ (for $C$ a
sufficiently large constant) \cite{Behmidi}, practical
algorithms require significantly larger $L$. This motivates the study
of restricted classes of tests. In this paper we study the class of
spectral (or eigenvalue-based) tests detecting the signal. Our
proof technique naturally allow to consider two further
generalizations of this problem that are of independent
interests. We briefly summarize our results below.

\vspace{0.2cm}

\noindent{\bf The Gaussian hidden clique problem.}
This is a special case of the above hypothesis testing setting,
whereby $F_0= \normal(0,1)$ and $F_1= \normal(1,1)$ (entries on the
diagonal are defined slightly differently for simplifying calculations).
Here and below $\normal(m,\sigma^2)$ denote the Gaussian distribution
of mean $m$ and variance
$\sigma^2$.  
Equivalently, let $\bZ$ be a random matrix from the Gaussian
Orthogonal Ensemble  (GOE) i.e. $\bZ_{ij}\sim \normal(0,1/n)$
independently for $i<j$, and $\bZ_{ii}\sim \normal(0,2/n)$.
Then, under hypothesis $H_{1,L}$ we have $\bX=
n^{-1/2}\bone_U\bone_U^{\sT}+\bZ$ ($\bone_U$ being the indicator vector
on $U$, and $|U|=L$),
and under hypothesis $H_{0}$, $\bX=\bZ$ (the factor $n$ in the
normalization is for technical convenience).

We then consider the following restricted hypothesis testing question.
Let $\lambda_1\ge \lambda_2\ge \cdots\ge \lambda_n$ be the ordered
eigenvalues of $\bX$. \emph{Is there a test that depends only on
  $\lambda_1,\dots,\lambda_n$ and that distinguishes $H_0$ from
  $H_{1,L}$ `reliably,' i.e. with error probability converging to $0$ as $n\to\infty$?}
Notice that the eigenvalues distribution does not depend on $U$ as
long as this is independent from the noise $\bZ$. We can therefore
think of $U$ as fixed for this question.

If $L\ge (1+\eps)\sqrt{n}$ then \cite{BBAP} implies that a simple test
checking whether $\lambda_1\ge 2+\delta$ for some $\delta
= \delta(\eps)>0$ is reliable. We prove that this result is tight, in
the sense that no spectral test is reliable for $L\le
(1-\eps)\sqrt{n}$.

\vspace{0.2cm}

\noindent{\bf Rank-one matrices in Gaussian noise.}
Our proof technique builds on a simple remark. Since the noise
$\bZ$ is invariant under orthogonal transformations\footnote{By this we mean
that, for any orthogonal matrix $\bR\in O(n)$, independent of $\bZ$,
$\bR\bZ\bR^{\sT}$ is distributed as $\bZ$.}, the above question is
equivalent to the following testing problem. For $\beta\in
\reals_{\ge 0}$, and $\bv\in\reals^n$, $\|\bv\|_2=1$
a uniformly random unit vector, test 
$H_0$: $\bX = \bZ$ versus $H_1$, $\bX = \beta\bv\bv^{\sT}+\bZ$. 
(The correspondence between the two problems yields $\beta = L/\sqrt{n}$.)

Again, this  problem (and a closely related asymmetric version
\cite{onatski2013asymptotic}) has been studied in the literature, and
it follows from \cite{BBAP} that a  reliable test exists for $\beta
\ge 1+\eps$. We provide a simple proof (based on the second moment
method) that no test is reliable for $\beta< 1+\eps$.

\vspace{0.2cm}

\noindent{\bf Rank-one tensors in Gaussian noise.} It turns our that
the same proof applies to an even more general problem: detecting a
rank-one signal in a noisy tensor. We carry out our analysis in this
more general setting for two reasons. First, we think that this
clarifies the what aspects of the model are important for our proof
technique to apply. Second, the problem estimating tensors from noisy data has
attracted significant interest recently within the machine learning
community \cite{Hsu,montanari2014statistical}. 

More precisely, we consider a noisy tensor $\bX\in\bigotimes^k\reals^n$, of the form $\bX =
\beta\,\bv\ok + \bZ$, where $\bZ$ is Gaussian noise, and $\bv$ is a
random unit vector. We consider the
problem of testing this hypothesis against $H_0$: $\bX =\bZ$. We
establish a threshold $\bnd_k$ such that no test  can be reliable for
$\beta<\bnd_k$ (in particular $\bnd_2=1$). We establish an analogous result for the asymmetric
case as well. 
Two differences are worth remarking for $k\ge 3$ with respect to the more familiar
matrix case $k=2$. First, we do not expect the second moment bound $\bnd_k$
to be tight, i.e. a reliable test to exist for all $\beta>\bnd_k$. On
the other hand, we can show that it is tight up to a universal ($k$
and $n$ independent) constant. Second, below $\bnd_k$ the problem is
more difficult than the matrix version below $\bnd_2=1$: not only no
reliable test exists but, asymptotically, any test behaves
asymptotically as random guessing.

\section{Main result for spectral detection}

Let $\bZ$ be a GOE matrix as defined in the previous
section. Equivalently if $\bG$ is an (asymmetric) matrix with i.i.d.
entries $\bG_{i,j}\sim\normal(0,1)$,
\begin{align}
\bZ = \frac{1}{\sqrt{2n}} \big(\bG+\bG^{\sT}\big)\, .
\end{align}
For a deterministic sequence of vectors $\bv(n)$, $\|\bv(n)\|_2 = 1$,
we consider the two hypotheses
\begin{align}
\begin{cases}H_0:\;\;\;& \bX = \bZ\, ,\\
H_{1,\beta}:\;\;\;& \bX = \beta \bv\bv^{\sT}+\bZ\, .
\end{cases}\label{eq:BetaProblem}
\end{align}
A special example is provided by the \emph{Gaussian hidden clique}
problem in which case $\beta= L/\sqrt{n}$ and $\bv = \bone_U/\sqrt{L}$
for 
some set $U\subseteq [n]$, $|U|=L$,
\begin{align}
\begin{cases}
H_0:\;\;\;& \bX = \bZ\, ,\\
H_{1,L}:\;\;\;& \bX = \frac{1}{\sqrt{n}} \bone_U\bone_U^{\sT}+\bZ\, .
\end{cases}\label{eq:LProblem}
\end{align}
Observe that the distribution of eigenvalues of $\bX$, under either
alternative, is invariant to
the choice of the vector $\bv$ (or subset $U$),  as long as the norm
of $\bv$ is kept fixed.
Therefore, any successful spectral algorithm will
distinguish between $H_0$ and $H_{1,\beta}$
but not give any information on the vector $\bv$
(or subset $U$, in the case
of $H_{1,L}$). 

We let $Q_0=Q_0(n)$ (respectively, $Q_1=Q_1(n)$)
denote the distribution of the eigenvalues of $\bX$
under $H_0$ (respectively $H_{1}=H_{1,\beta}$ or $H_{1,L}$).

A \emph{spectral statistical test} for distinguishing between
$H_0$ and $H_{1}$ (or simply a spectral test) is a
measurable map $T_n:(\lambda_1,\ldots,\lambda_n)\mapsto \{0,1\}$.
To formulate precisely what we mean by the word \textit{distinguish}, we
introduce the following notion.
\begin{definition}
        \label{def-cont}
For each $n\in\naturals$, let $\prob_{0,n}$, $\prob_{1,n}$ be two
probability measures on  the same measure space $(\Omega_n,\cF_n)$.
We say that the sequence $(\prob_{1,n})$ is contiguous with respect to
$(\prob_{0,n})$ if, for any sequence of events $A_n\in \cF_n$,
\begin{align}
\lim_{n\to\infty} \prob_{0,n}(A_n) = 0
\;\;\Rightarrow\;\;\lim_{n\to\infty} \prob_{1,n}(A_n) = 0 \, .
\end{align}
\end{definition}
\noindent Note that contiguity is not in general a symmetric relation.

In the context of the spectral statistical tests described above, the sequences
$A_n$ in Definition \ref{def-cont} (with $P_n=Q_0(n)$ and
$Q_n=Q_{1}(n)$) can be put in correspondence
with spectral
statistical tests $T_n$ by taking $A_n=\{(\lambda_1,\ldots,\lambda_n):
T_n(\lambda_1,\ldots,\lambda_n)=0\}$. We will thus say that
$H_{1}$ is \textit{spectrally contiguous} with respect to $H_0$
if $Q_n$ is contiguous with respect to $P_n$.

Our main result on the Gaussian hidden clique problem is the following.
\begin{theorem}\label{thm:main}
For any sequence
  $L=L(n)$ satisfying $\limsup_{n\to\infty} L(n)/\sqrt{n}<1$,
  the hypotheses $H_{1,L}$ are
  spectrally contiguous with respect to $H_0$.

Equivalently for any sequence of vectors $\bv(n)$, and any
$\beta=\beta(n)$ satisfying $\limsup_{n\to\infty} \beta(n)<1$,
  the hypotheses $H_{1,\beta}$ are
  spectrally contiguous with respect to $H_0$.
\end{theorem}

Several remarks are in order with respect to Definition~\ref{def-cont}
and Theorem~\ref{thm:main}. First, we rule out arbitrary spectral
tests - not just tests that have polynomial running time (in
$n$). Second, we \emph{do not} rule out the existence of a spectral
test (or even a spectral test running in polynomial time) that
distinguishes between $H_0$ and $H_{1,L}$ with total probability of error
$\prob_0(T=1)+\prob_1(T=0)\le 1-\delta$. Indeed, as discussed in Section
\ref{sec:Remarks}, it is quite easy to construct such a test.

Finally Theorem~\ref{thm:main} provides an example of two family of
distributions $H_0$ and $H_{1,L}$ such that the total variation distance
between $H_0$ and $H_{1,L}$ is $1-o_n(1)$ whereas
$H_0$ and $H_{1,L}$ are spectrally contiguous.

\subsection{Contiguity and integrability}
Contiguity is related to a notion of uniform  absolute
continuity of measures.
Recall that a
probability measure
$\mu$ on a measure space
is \emph{absolutely continuous} with respect to another probability measure
$\nu$ if for every measurable set $A$, $\nu(A)=0$ implies that $\mu(A)=0$,
in which case
there exists a $\nu$-integrable, non-negative function $f\equiv \frac{\de\mu}{\de\nu}$
(the
\emph{Radon-Nikodym derivative} of $\mu$
with respect to $\nu$), so that
$\mu(A)=\int_{A} f \, \de\nu$
for every measurable set $A$.
We then have the following known useful fact, which will be the basis for
proving contiguity, and whose proof  is given
for completeness.
\begin{lemma}\label{lemma:EasyLemma}
        Within the setting of Definition \ref{def-cont},
        assume that $\prob_{1,n}$
is absolutely continuous with respect to $\prob_{0,n}$, and denote by
$\Lambda_n\equiv \frac{\de \prob_{1,n}}{\de\prob_{0,n}}$ its
Radon-Nikodym derivative.

\noindent
(a) If $\lim\sup_{n\to\infty}\E_{0,n}(\Lambda_n^2)<\infty$, then $(\prob_{1,n})$ is contiguous with respect to
$(\prob_{0,n})$. \\
(b) If $\lim_{n\to\infty}\E_{0,n}(\Lambda_n^2) = 1$, then $\lim_{n\to\infty}\|\prob_{0,n}-\prob_{1,n}\|_{\TV}
= 0$, where $\|\cdot\|_{\TV}$ denotes the total variation distance, i.e.
$$\|\prob_{0,n}-\prob_{1,n}\|_{\TV}\equiv\sup_A |\prob_{0,n}(A)-\prob_{1,n}(A)\|.$$
\end{lemma}

\begin{proof}
(a) Let $A_n$ be any sequence of events
such that $\prob_{0,n}(A_n)\to 0$.
Then
\begin{align}
\prob_{1,n}(A_n) = \E_{0,n}\{\Lambda_n\bone_{A_n}\} \le
\E_{0,n}\{\Lambda_n^2\}^{1/2}\prob_{0,n}(A_n)^{1/2}\to 0\, ,
\end{align}
where the last limit follows since $\E_{0,n}\{\Lambda_n^2\}\le C$
for all $n$ large enough.

(b) Note that $\E_{0,n}\Lambda_{n} =1$,
whence
\begin{align}
2\|\prob_{0,n}-\prob_{1,n}\|_{\TV} = \E_{0,n}\{|\Lambda_n-1|\} \le
\E_{0,n}\{(\Lambda_n-1)^2\}^{1/2} = \sqrt{\E_{0,n}(\Lambda_n^2)-1}\, .
\end{align}
\end{proof}

\subsection{Method and structure of the paper}

Consider problem (\ref{eq:BetaProblem}).
We use the fact that the law of the eigenvalues under
both $H_0$ and $H_{1,\beta}$ are invariant under conjugations
by a orthogonal matrix.
Once we conjugate matrices sampled under the hypothesis
$H_{1,\beta}$ by an independent
orthogonal matrix sampled according to the Haar distribution,
we get a matrix distributed as
\begin{align}
\bX = \beta\bv\bv^{\sT}+\bZ\, ,
\end{align}
where $\bu$ is uniform on the $n$-dimensional sphere,
and $\bZ$ is a GOE matrix
(with off-diagonal entries of
variance $1/n$).
Letting $\prob_{1,n}$ denote the law of
$\beta \bu\bu^{\sT}+\bZ$ and
$\prob_{0,n}$ denote the law of $\bZ$,
we show that $\prob_{1,n}$
is contiguous with respect to $\prob_{0,n}$,
which implies
that the law of eigenvalues $Q_{1}(n)$ is  contiguous with respect to $Q_0(n)$.

To show the contiguity, we consider a more general setup,
of independent interest,
of
Gaussian tensors of order $k$,
and in that setup  show that
the Radon-Nikodym derivative $\Lambda_{n,L}=\frac{\de\prob_{1,n}}{\de \prob_{0,n}}$ is
uniformly square integrable under $\prob_{0,n}$; an
application of Lemma
\ref{lemma:EasyLemma} then quickly yields Theorem \ref{thm:main}.

The structure of the paper is as follows. In the next section, we
define formally the detection problem for a symmetric tensor of order
$k\geq 2$. We
show the existence of a threshold under which detection is not possible
(Theorem \ref{main:Tensor}), and show how Theorem \ref{thm:main} follows
from this.
Section
\ref{sec-proofthemTensor} is devoted to the proof of Theorem
\ref{main:Tensor}, and concludes with some additional remarks and
consequences of Theorem \ref{main:Tensor}.
Section \ref{sec-asym} treats the case of asymmetric Gaussian tensors.
Finally,
Section \ref{sec-relwork} is devoted to a description of
the relation between the Gaussian hidden clique problem and hidden clique problem
in computer science, and related literature.

\section{A symmetric tensor model and a reduction}
Exploiting rotational invariance,
we will reduce the spectral detection problem to a detection problem
involving a standard detection problem between random matrices.
Since the latter
generalizes to a tensor setup, we first introduce a general
Gaussian hypothesis testing for $k$-tensors, which is of independent
interest. We then explain how
the spectral detection problem reduces to the special case of $k=2$.

\subsection{Preliminaries and notation}
We use lower-case boldface for vectors
(e.g. $\bu$, $\bv$, and so
on) and upper-case boldface for matrices and tensors (e.g. $\bX,\bZ$,
and so on).
The ordinary scalar product and $\ell_p$ norm over vectors are denoted
by $\<\bu,\bv\> = \sum_{i=1}^n\bu_i\bv_i$, and $\|\bv\|_p$.
We write
$\bbS^{n-1}$ for the unit sphere in $n$ dimensions
\begin{align}
\bbS^{n-1} \equiv\big\{\bx\in\reals^n:\; \;\|\bx\|_2=1\big\}\, .
\end{align}

Given $\bX \in \bigotimes^k \reals^n$ a real $k$-th order tensor,
we let $\{\bX_{i_1,\dots,i_k}\}_{i_1,\dots,i_k}$ denote its coordinates.
The outer product of two tensors is $\bX\otimes \bY$, and, for
$\bv \in \reals^n$, we define $\bv^{\otimes k} = \bv \otimes \cdots
\otimes \bv \in \bigotimes^k\reals^n$
as the $k$-th outer power of $\bv$. We define the inner product of two tensors $\bX, \bY \in \bigotimes^k \reals^n$ as
\begin{align}
 \< \bX , \bY \> = \sum_{i_1, \cdots ,i_k \in [n] } \bX_{i_1, \cdots
  ,i_k} \bY_{i_1, \cdots ,i_k}~~.
\end{align}
We define the Frobenius (Euclidean) norm of a tensor $\bX$  by
$\|\bX\|_F = \sqrt{\<\bX,\bX\>}$, and its operator
norm by
\begin{align}
\|\bX\|_{op} \equiv\max \{ \< \bX,\bu_1\otimes \cdots \otimes \bu_k
\>~:~\forall i\in [k]~,~\|\bu_i\|_2\leq 1 \}.
\end{align}
It is easy to check that this is indeed a norm. For the special case
$k=2$, it reduces to the ordinary $\ell_2$ matrix operator norm
(equivalently, to the largest singular value of $\bX$).

For a permutation $\pi \in \mathfrak S_k$,
we will denote by $\bX^\pi$ the tensor
with permuted indices $\bX^\pi_{i_1, \cdots, i_k} = \bX_{\pi(i_1),
\cdots , \pi(i_k)}$. We call the tensor $\bX$ \emph{symmetric} if,
for any permutation $\pi \in \mathfrak S_k$, $\bX^\pi = \bX$.
It is proved \cite{Waterhouse90} that, for symmetric tensors, for symmetric tensors we have the equivalent
representation
\begin{align}
\|\bX\|_{op} \equiv\max \{ |\< \bX,\bu\ok\>|~:~~~\|\bu\|_2\leq 1 \}.
\end{align}
We define $\overline\reals \equiv \reals \cup \infty$ with the usual conventions of arithmetic operations.
\subsection{The symmetric tensor model and main result}

We denote by  $\bG \in \bigotimes^k \reals^n$  a tensor with
independent and identically distributed entries
$\bG_{i_1, \cdots, i_k}\sim\normal(0,1)$ (note that this tensor is not
symmetric).

We define the \emph{symmetric standard normal} noise tensor
$\bZ \in \bigotimes^k \reals^n$
by

\begin{align}\label{eq:symNoiseDefinition}
\bZ =\frac{1}{k!  }\sqrt{\frac2n
} \sum_{\pi \in \mathfrak S_k }
\bG^\pi\, .
\end{align}
Note  that the subset of entries with unequal indices form an i.i.d.
collection $\{\bZ_{i_1,i_2,\dots,i_k} \}_{i_1<\dots<i_k} \sim \normal(
0,2/(n(k!)))$.

With this normalization, we  have, for any symmetric tensor $\bA\in \bigotimes^k\reals^n$
\begin{align}
\E\big\{e^{\< \bA,\bZ \>}\big\} =\exp\Big\{\frac{1}{n} \|\bA\|_F^2\Big\}\, .\label{eq:ExpTensor}
\end{align}
We will also use the fact that $\bZ$ is invariant in distribution under conjugation by
orthogonal transformations, that is, that for any orthogonal matrix
$U\in O(n)$,
$\{\bZ_{i_1,\ldots,i_k}\}$
has the same distribution as
$\{\sum_{j_1,\ldots,j_k} \left(\prod_{\ell=1}^k U_{i_\ell,j_\ell}\right)\cdot
\bZ_{j_1,\ldots,j_k}\}$.

Given a parameter $\beta\in\reals_{\ge 0}$, we consider the following model for a random symmetric tensor
$\bX$:
\begin{align}
\bX \equiv \beta \,\bv\ok + \bZ\,,\label{eq:SymmetricModel}
\end{align}
with $\bZ$ a standard normal tensor, and $\bv$ uniformly
distributed over the unit sphere $\bbS^{n-1}$.
In the case $k=2$ this is the standard
rank-one deformation of a GOE matrix.

We let $\prob_{\beta}=\prob_{\beta}^{(k)}$ denote the law of
$\bX$ under model (\ref{eq:SymmetricModel}).
\begin{theorem}\label{main:Tensor}
For $k\ge 2$, let
\begin{align}
\bnd_k\equiv \inf_{q\in (0,1)} \sqrt{-\frac{1}{q^k}\log(1-q^2)}\, .
\end{align}
Assume $\beta<\bnd_k$. Then, for any $k\ge 3$, we have
\begin{align}
\lim_{n\to\infty}\big\|\prob_{\beta}-\prob_0\big\|_{\TV} = 0\, .
\end{align}
Further, for $k=2$ and $\beta<\bnd_k=1$, $\prob_{\beta}$ is contiguous
with respect to $\prob_0$.
\end{theorem}
\noindent
The notation $\bnd_k$ refers to the fact that this is the threshold for
the second moment method to work.

\subsection{Reduction of spectral detection to the symmetric
        tensor model, $k=2$, and proof of Theorem \ref{thm:main}}

Recall that in the setup of Theorem \ref{thm:main},
$Q_{0,n}$ is the law of the eigenvalues of $\bX$
under $H_{0}$ and $Q_{1,n}$ is the law of the eigenvalues of $\bX$
under $H_{1,L}$.
Then $Q_{1,n}$ is invariant by conjugation of orthogonal matrices.
Therefore, the detection problem is not changed if we replace
$\bX= n^{-1/2}\bone_U\bone_U^{\sT} +\bZ$ by 
\begin{align}
\widehat{\bX}\equiv \bR\bX\bR^{\sT} = \frac{1}{\sqrt{n}}
\bR\bone_U(\bR\bone_U)^{\sT} + \bR\bZ\bR^{\sT}\, ,
\end{align}
 where $\bR\in O(n)$ is an orthogonal matrix sampled according to the Haar measure.
A direct calculations yields
\begin{align}
\widehat{\bX}=\beta\bu \bu^{\sT}+\widetilde{\bZ},
\end{align}
where $\bu$ is uniform on the $n$ dimensional sphere,
$\beta=L/\sqrt{n}$,
and $\widetilde{\bZ}$ is a GOE matrix (with off-diagonal entries of
variance $1/n$).
Furthermore, $\bu$ and $\widetilde{\bZ}$ are independent of one another.

Let $\prob_{1,n}$ be the law of $\widehat{\bX}$.
Note that $\prob_{1,n}=\prob_\beta^{(k=2)}$ with
$\beta=L/\sqrt{n}$.
We can relate the detection problem of $H_0$ vs. $H_{1,L}$ to
the detection problem of $\prob_{0,n}$ vs. $\prob_{1,n}$ as follows.
\begin{lemma}\label{lem:second_step_a}
        (a) If $\prob_{1,n}$ is contiguous with respect
        to $\prob_{0,n}$ then $H_{1,L}$ is spectrally contiguous with respect to
        $H_0$.\\
        (b) We have
        $$\|Q_{0,n}-Q_{1,n}\|_{\TV}\leq \|\prob_{0,n}-\prob_{1,n}\|_{\TV}.$$
\end{lemma}
\begin{proof}
        $(a)$
Let $\mathcal{F}$ be the $\sigma$-algebra generated by all rotation-invariant
events,
and let $\mathcal{G}$ be the $\sigma$-algebra generated by the ordered eigenvalues. Then $\mathcal{F}=\mathcal{G}$.

Let $A_n\in \mathcal{G}$ be (spectral)
events so that $Q_{0,n}(A_n)\to_{n\to\infty} 0$. Note that (with some abuse
of notation) we have that $Q_{0,n}$ coincides with $\prob_{0,n}$ and
$Q_{1,n}$ coincides with $\prob_{1,n}$ on ${\cal G}$.
Then $\prob_{1,n}(A_n)\to_{n\to\infty} 0$ by assumption and therefore
$Q_{1,n}(A_n)=\prob_{1,n}(A_n)\to_{n\to\infty} 0$, proving the
contiguity of $Q_{1,n}$ with respect to $Q_{0,n}$.

\noindent
(b) We have
\begin{eqnarray}
        \label{eq-display}
        \sup_A{|\prob_{0,n}(A)-\prob_{1,n}(A)|} &\geq&
        \sup_{A \in \mathcal{F}}{|\prob_{0,n}(A)-\prob_{1,n}(A)|}=
        \sup_{A \in \mathcal{G}}{|\prob_{0,n}(A)-\prob_{1,n}(A)|}\nonumber \\
        &=&
        \sup_{A \in \mathcal{\mathcal{B}}}{|Q_{1,n}(A)-Q_{0,n}(A)|},
\end{eqnarray}
where $\mathcal{B}$ is the Borel $\sigma$-algebra in
$\Sigma^n:=\{\lambda_1 \geq \lambda_2 \ldots \geq \lambda_n\}$.
\end{proof}
\noindent
{\bf Remark:} The first inequality in \eqref{eq-display} (and hence
the inequality in point $(b)$ of the statement) is actually an
equality, but we will not use this fact.

In view of Lemma
\ref{lem:second_step_a}, Theorem  \ref{thm:main}
is an immediate consequence of Theorem \ref{main:Tensor}.

\section{Proof of Theorem \ref{main:Tensor}}
\label{sec-proofthemTensor}
The proof uses the following large deviations lemma, which follows,
for instance, from \cite[Proposition 2.3]{dembo2014matrix}.
\begin{lemma}\label{lemma:LargeDev}
Let $\bv$ a uniformly random vector on the unit sphere $\bbS^{n-1}$ and let
$\<\bv,\be_1\>$ be its first coordinate. Then, for any interval $[a,b]$ with
$-1\le a<b\le 1$
%
\begin{align}
\lim_{n\to\infty} \frac{1}{n}\log \prob(\<\bv,\be_1\>\in [a,b]) = \max\Big\{ \frac{1}{2}\log(1-q^2):\, q\in [a,b] \,\Big\}\, .
\end{align}
\end{lemma}
\begin{proof}[Proof of Theorem \ref{main:Tensor}]
We denote by $\Lambda$ the Radon-Nikodym derivative of $\prob_{\beta}$ with respect to $\prob_0$.
By definition $\E_0\Lambda = 1$.
It is easy to derive the following formula
\begin{align}
\Lambda =\int \exp\Big\{-\frac{n\beta^2}{4} + \frac{n\beta}{2}\,\<\bX,\bv\ok\>\Big\} \, \mu_n(\de \bv)\, .
\end{align}
where $\mu_n$ is the uniform measure on $\bbS^{n-1}$. Squaring and using (\ref{eq:ExpTensor}), we get
\begin{align}
\E_0\Lambda^2 & = e^{-n\beta^2/2}\int\E_0
\exp\Big\{  \frac{n\beta}{2}\<\bX,\bv_1\ok+\bv_2\ok\>\Big\} \, \mu_n(\de \bv_1)\mu_n(\de\bv_2)\nonumber\\
& = e^{-n\beta^2/2}\int \exp\Big\{\frac{n\beta^2}{4}\big\|\bv_1\ok+\bv_2\ok\big\|_F^2\Big\}\mu_n(\de \bv_1)\mu_n(\de\bv_2) \nonumber\\
& = \int \exp\Big\{\frac{n\beta^2}{2}\<\bv_1,\bv_2\>^k\Big\} \, \mu_n(\de\bv_1)\mu_n(\de\bv_2)\nonumber\\
& = \int \exp\Big\{\frac{n\beta^2}{2}\<\bv,\be_1\>^k\Big\} \, \mu_n(\de \bv)\label{eq:SecondMomentCalculation}
\, ,
\end{align}
where in the first step we used \eqref{eq:ExpTensor} and in the last step, we used rotational invariance.

Let $F_{\beta}:[-1,1]\to \overline\reals$ be defined by
\begin{align}
F_{\beta}(q) \equiv \frac{\beta^2q^k}{2} +\frac{1}{2}\log(1-q^2)\,.
\end{align}
Using Lemma \ref{lemma:LargeDev}, for any $-1\le a< b\le 1$,
\begin{align}
\int \exp\Big\{\frac{n\beta^2}{2}\<\bv,\be_1\>^k\Big\} \,\ind(\<\bv,\be_1\> \in[a, b])\, \mu_n(\de \bv)
= \exp\Big\{ n\max_{q\in [a,b]} F_{\beta}(q)+ o(n)\Big\}\, .
\end{align}
It follows from the definition of $\bnd_k$ that $\max_{|q|\ge \eps} F_{\beta}(q)<0$ for any $\eps>0$.
Hence
\begin{align}
\E_0\Lambda^2 &  \le \int \exp\Big\{\frac{n\beta^2}{2}\<\bv,\be_1\>^k\Big\} \, \ind(|\<\bv,\be_1\>| \le \eps)\,
\mu_n(\de \bv) + e^{-c(\eps)n}\, ,
\end{align}
for some $c(\eps)>0$ and all $n$ large enough.
Next notice that, under $\mu_n$, $\<\bv,\be_1\> \ed G/(G^2+Z_{n-1})^{1/2}$ where $G\sim\normal(0,1)$ and
$Z_{n-1}$ is
a $\chi^2$ with $n-1$ degrees of freedom independent of $G$.
Then, letting $Z_n \equiv G^2+Z_{n-1}$ (a $\chi^2$ with $n$ degrees of freedom)
\begin{align}
\E_0\Lambda^2 & \le  \E \Big\{\exp\Big(\frac{n\beta^2}{2}\frac{|G|^k}{Z_n^{k/2}}\Big) \, \ind(|G/Z_n^{1/2}| \le \eps)\Big\}+ e^{-c(\eps)n}\nonumber\\
& \le  \E \Big\{\exp\Big(\frac{n\beta^2}{2}\frac{|G|^k}{Z_n^{k/2}}\Big) \, \ind(|G/Z_n^{1/2}| \le \eps)\, \ind(Z_{n-1}\ge n(1-\delta))\Big\}\nonumber\\
&\phantom{=}+e^{n\beta^2\eps^k/2}\prob\big\{Z_{n-1}\le n(1-\delta)\big\}+ e^{-c(\eps)n}\nonumber\\
&\le   \E \Big\{\exp\Big(\frac{n^{1-(k/2)}\beta^2}{2(1-\delta)^{k/2}}\, |G|^k\Big)
\ind(|G|^2\leq 2\eps n) \, \Big\}
+e^{n\beta^2\eps^k/2}\prob\big\{Z_{n-1}\le n(1-\delta)\big\}+ e^{-c(\eps)n}
\nonumber\\
&=
\frac{2}{\sqrt{2\pi}}\int_0^{2\eps n}e^{C(\beta,\delta)
  n^{1-k/2}x^k-x^2/2}\de x+e^{n\beta^2\eps^k/2}\prob\big\{Z_{n-1}\le n(1-\delta)\big\}+ e^{-c(\eps)n} \,,
\label{eq-long}
\end{align}
where $C(\beta,\delta)=\beta^2/(2(1-\delta)^{k/2})$. Now, for any $\delta>0$, we can (and will) choose $\eps$ small enough so that both
$e^{n\beta^2\eps^k/2}\prob\big\{Z_{n-1}\le n(1-\delta)\big\} \to 0$ exponentially fast (by tail bounds on $\chi^2$ random variables) and, if $k\geq 3$,
the argument of the exponent  in  the integral in the right hand side of \eqref{eq-long} is bounded above by $-x^2/4$, which is possible since the argument  vanishes at $x^*=2C(\beta,\delta)n^{1/2}$.
Hence, for any $\delta>0$, and all $n$ large enough, we have
\begin{align}
\E_0\Lambda^2 \le\frac{2}{\sqrt{2\pi}}\int_0^{2\eps
  n}e^{C(\beta,\delta) n^{1-k/2}x^k-x^2/2}\de x
+e^{-c(\delta)n}\,,
\label{eq-newlast}
\end{align}
for some $c(\delta)>0$.

Now, for $k\ge 3$ the  integrand in  \eqref{eq-newlast} is dominated by $e^{-x^2/4}$ and converges pointwise
(as $n\to\infty$) to $1$. Therefore,
since $\E_0\Lambda^2\ge (\E_0\Lambda)^2=1$,
\begin{align}
k\ge 3:\;\;\;\;\; \lim_{n\to\infty} \E_0\Lambda^2 = 1\, .
\end{align}
For $k=2$, the argument is independent of $n$ and can be integrated immediately, yielding (after taking the limit $\delta\to 0$)
\begin{align}
k=2:\;\;\;\;\; \lim\sup_{n\to\infty} \E_0\Lambda^2 \le \frac{1}{\sqrt{1-\beta^2}}\, .
\end{align}
 (Indeed,  the above calculation implies that the limit exists and is given by the right-hand side.)

The proof is completed by invoking Lemma  \ref{lemma:EasyLemma}.
\end{proof}

\subsection{Some remarks and consequences}
\label{sec:Remarks}

\noindent {\bf Threshold values.}
In the table below we report the numerical values of $\bnd_k$ for
a few values of $k$.  The exact value $\bnd_k=1$  for the
matrix case $k=2$ follows from $\log(1-q^2)\le -q^2$.

\vspace{0.25cm}
\begin{center}

\begin{tabular}{c|c}
$k$ & $\bnd_k$\\
\hline
$2$ & $1$ \\
$3$ & $1.398841$\\
$4$ & $ 1.566974$\\
$5$ & $1.67676$\\
$6$ & $1.757589$\\
$10$ & $1.955118$\\
$100$ & $2.595874$\\
\end{tabular}

\end{center}
\vspace{0.25cm}

Also, it is not difficult to derive the asymptotics $\bnd_k =
\sqrt{\log(k/2)} +o_k(1)$ for large $k$.

\vspace{0.25cm}

\noindent {\bf Tightness of the threshold values.} As mentioned in the
introduction, for $k=2$ and $\beta>1$, it is known that the largest
eigenvalue of $\bX$, $\lambda_1(\bX)$ converges almost surely to
$(\beta+1/\beta)$ \cite{BBAP}.
As a consequence $\|\prob_{0}-\prob_{\beta}\|_{\TV}\to 1$ for all
$\beta>1$:
the second moment bound is tight.

For $k\ge 3$, it follows by the
triangular inequality that $\|\bX\|_{op}\ge
\beta-\|\bZ\|_{op}$,
and further $\lim\sup_{n\to\infty}\|\bZ\|_{op}\le \mu_k$ almost surely
as $n\to\infty$ \cite{talagrand2006free,Auffinger13} for some bounded
$\mu_k$.
It follows that $\|\prob_{0}-\prob_{\beta}\|_{\TV}\to 1$ for all
$\beta>2\mu_k$ \cite{montanari2014statistical}.
Hence, the second moment bound is off by a $k$-dependent factor.
For large $k$,  $2\mu_k = \sqrt{2\log k} +O_k(1)$ and hence the factor is
indeed bounded in $k$.
\textcolor{red}{}

\vspace{0.25cm}

\noindent {\bf Behavior below the threshold.} Let us stress an
important qualitative difference between $k=2$ and $k\ge 3$,
for $\beta<\bnd_k$.
For $k\ge 3$, the two models are indistinguishable and any test is
essentially as good as random guessing. Formally, for any
measurable function $T:\otimes^k\reals^n\to\{0,1\}$, we have
\begin{align}
\lim_{n\to\infty}\big[\prob_0(T(\bX) = 1) + \prob_{\beta}(T(\bX) =
0)\big] =1\, .
\end{align}
For $k=2$, our result implies that, for $\beta<1$,
$\|\prob_{0}-\prob_{\beta}\|_{\TV}$ is bounded away from $1$. On the
other hand, it is easy to see that it is bounded away from $0$ as
well, i.e.
\begin{align}
0<\lim\inf_{n\to\infty}\|\prob_{0}-\prob_{\beta}\|_{\TV}\le
\lim\sup_{n\to\infty}\|\prob_{0}-\prob_{\beta}\|_{\TV}<1\, .
\end{align}
Indeed, consider for instance the statistics $S = \trace(\bX)$. Under
$\prob_0$, $S\sim \normal(0,2)$, while under $\prob_{\beta}$, $S\sim
\normal(\beta,2)$.
Hence
\begin{align}
\lim\inf_{n\to\infty}\|\prob_{0}-\prob_{\beta}\|_{\TV}\ge
\|\normal(0,1)-\normal(\beta/\sqrt{2},1)\|_{\TV}
= 1-2\Phi\Big(-\frac{\beta}{2\sqrt{2}}\Big)>0
\end{align}
(Here $\Phi(x) =\int_{-\infty}^x e^{-z^2/2}\de z/\sqrt{2\pi}$ is the
Gaussian distribution function.)
The same phenomenon for rectangular matrices ($k=2$) is discussed in detail in
\cite{onatski2013asymptotic}.

%

\section{Asymmetric tensor model}
\label{sec-asym}

As before, we  denote by  $\bG \in \bigotimes^k \reals^n$  a tensor with
independent and identically distributed entries
$\bG_{i_1, \cdots, i_k}\sim\normal(0,1)$.
We define the \emph{asymmetric standard normal} noise tensor $\bZ \in \bigotimes^k \reals^n$
by
\begin{align}\label{eq:asymNoiseDefinition}
\bZ =\sqrt{\frac 1 n} \,
\bG\, .
\end{align}
In particular, all entries are i.i.d. $\bZ_{i_1,\dots,i_k}\sim
\normal(0,1/n)$. With this normalization, we have of course
\begin{align}
\E\big\{e^{\< \bA,\bZ \>}\big\} =\exp\Big\{\frac{1}{2n} \|\bA\|_F^2\Big\}\, .\label{eq:ExpTensorAsymm}
\end{align}
Given $\lambda\in\reals_{\ge 0}$, we consider observations
$\bX\in\bigotimes^{k}\reals^n$ given by:
\begin{align}
\bX \equiv \lambda \,\bv_1\otimes\bv_2\otimes\cdots\otimes\bv_k + \bZ\,,\label{eq:AsymmetricModel}
\end{align}
with $\bZ$ a an asymmetric standard normal tensor, and $\bv_1$,\dots,
$\bv_k$ independent and uniformly
distributed over the unit sphere $\bbS^{n-1}$. In the case $k=2$ we
recover, again, the classical spiked model. Note that a further layer
of generalization would be obtained by considering a `rectangular'
tensor
$\bX\in \reals^{n_1}\otimes\cdots\otimes \reals^{n_k}$. We prefer to
assume $n_1=\dots = n_k$ to limit inessential technical complications.

We denote by $\prob_{\lambda}=\prob_{\lambda}^{(k)}$ the law of $\bX$
under the model (\ref{eq:AsymmetricModel})
\begin{theorem}
For $k\ge 2$, let $\lnd_k
\equiv(k/2)^{1/2}\bnd_{k}$, i.e.
\begin{align}
\lnd_k \equiv \inf_{q\in (0,1)}
\sqrt{-\frac{k}{2q^k}\,\log(1-q^2)}\, .
\end{align}
Assume $\lambda<\lnd_k$. Then, for any $k\ge 3$, we have
\begin{align}
\lim_{n\to\infty}\|\prob_{\lambda}-\prob_0\|_{\TV} = 0\, .
\end{align}
Further, for $\lambda<\lnd_2 = 1$, $\prob_{\lambda}$ is
contiguous with respect to $\prob_0$.
\end{theorem}
\begin{proof}
The proof is very similar to the one of Theorem \ref{main:Tensor}. We
will therefore limit ourselves to outline the main steps, and the
differences with respect to the symmetric case.
First we compute the Radon-Nikodym derivative of $\prob_{\lambda}$
with respect to $\prob_0$:
\begin{align}
\Lambda = \int \exp\Big\{-\frac{n\lambda^2}{2} + n\lambda\,\<\bX,\bv_1\otimes\cdots\otimes\bv_k\>\Big\} \,\mu^{\otimes
k}_n(\de \bv)\, .
\end{align}
Here we introduced the notation $\mu^{\otimes k}_n(\de \bv)  = \mu_n(\de \bv_1)\cdots\mu_n(\de\bv_k)$
Proceeding as in the proof of Eq.~(\ref{eq:SecondMomentCalculation}),
we get
\begin{align}
\E_0\Lambda^2 =  \int
\exp\Big\{n\lambda^2\prod_{i=1}^k\<\bv_i,\be_1\>^k\Big\} \, \mu^{\otimes
k}_n(\de \bv)\, .\label{eq:2ndMomentAsymmetric}
\end{align}
Now, let $G_{\lambda}:[-1,1]^k\to \overline\reals$ be defined by
\begin{align}
G_{\lambda}(q_1,q_2,\dots,q_k) = \lambda^2\prod_{i=1}^k q_i\,
+\frac{1}{2}\sum_{i=1}^k \log(1- q_i^2)\, .
\end{align}
Invoking again Lemma \ref{lemma:LargeDev}, we obtain that for any
open set $J\in [-1,1]^k$
\begin{align}
\int
\exp\Big\{n\lambda^2\prod_{i=1}^k\<\bv_i,\be_1\>^k\Big\} \,
\ind\Big((\<\bv_1,\be_1\>,\dots,\<\bv_k,\be_1\>)\in J\Big)\mu^{\otimes
k}_n(\de \bv) = \exp\Big\{n\max_{\bq\in J}  G_{\lambda}(\bq) + o(n)\Big\}\, .
\end{align}
The key observation is that, for $\lambda<\lnd_k$, we have
$\max_{\bq\in [-1,1]^k} G_{\lambda}(\bq) =0$, with the maximum being
uniquely achieved at $\bq =0$. Once this claim is proved,
we can restrict the integral  (\ref{eq:2ndMomentAsymmetric}) to a
neighborhood of $(\<\bv_1,\be_1\>,\cdots,\<\bv_k,\be_1\>) =0$ and
obtain by an argument completely analogous to the symmetric case
%
%
$$k\ge 3\Rightarrow \lim_{n\to\infty} \E_0\Lambda^2 = 1\,.
\quad
k=2    \Rightarrow \lim\sup_{n\to\infty} \E_0\Lambda^2 \le \frac{1}{\sqrt{1-\lambda^4}}\,.
%
$$

To prove the above claim, and hence complete the proof, note that:
$\bq=0$ is a local maximum of $G_{\lambda}(\bq)$;
$G_{\lambda}(|q_1|,\dots,|q_k|)\ge G_{\lambda}(q_1,\dots,q_k)$;
$G_{\lambda}(\bq)\to -\infty$ if $\|\bq\|_{\infty}\to 1$. It is
therefore sufficient to prove that any other local maximum $\bq_*\in
[0,1)^k\setminus \{0\}$ has $G_{\lambda}(\bq_*)<0$. The stationarity
condition of $G_{\lambda}$ reads
\begin{align}
\lambda^2\prod_{j\in[k]\setminus i}q_j = \frac{q_i}{1-q_i^2}\,
,\;\;\forall i\in[k]\, ,
\end{align}
whence (after multiplying by $q_i$)
\begin{align}
\frac{q_1^2}{1-q_1^2} = \frac{q_2^2}{1-q_2^2} = \dots =
\frac{q_k^2}{1-q_k^2} \, .
\end{align}
since $x\mapsto x^2/(1-x^2)$ is strictly monotone increasing on
$(0,1]$, we deduce that $q_1=q_2=\dots= q_k$. It is therefore
sufficient to check $G_{\lambda}(q,q,\dots,q)< 0$ for all $q\in
(0,1)$. This is guaranteed by $\lambda<\lnd_k$.
\end{proof}

\section{Related work}
\label{sec-relwork}

Detection problems with similar flavor to the Gaussian hidden clique problem have been studied over the years in several fields including computer science, physics and statistics. Typically, in such problems there is ``planted" object with special properties along with random noise which makes the detection of the planted object a nontrivial task.
In the classical $G(n,1/2)$ planted clique problem,
the computational problem is to find the planted clique (of cardinality $k$) efficiently (e.g., in polynomial time) where we assume the location of the planted clique is hidden and is not part of the input. There are several algorithms that recover the planted clique in polynomial time when $k=C\sqrt{n}$ where $C>0$ is a constant independent of $n$ \cite{AlKS,Dekel,Feige1,Feige3}. In \cite{Desh} it is proven that a planted clique can be recovered in time $O(n^2\log(n))$, whenever $C=e^{-1/2}+\epsilon$. The work of \cite{AlKS} demonstrates that it is possible to find a planted clique of size $c\sqrt{n}$ in time $n^{O(\log(1/c))}$. Despite significant effort, no polynomial time algorithm for this problem is known when $k=o(\sqrt n).$
In the decision version of the planted clique problem, one
seeks an efficient algorithm that distinguishes between a random graph
distributed as $G(n,1/2)$ or a random graph containing a planted clique of size $k \geq(2+\delta)\log n$ (for $\delta>0$; the natural threshold for the problem
is the size of the largest clique in a random sample of
$G(n,1/2)$, which is asymptotic to $2\log n$ \cite{grimmett}).
No polynomial time algorithm is known for this
decision problem if $k=o(\sqrt{n})$. There are several hardness results for computational problems in game theory \cite{Minder} (see also \cite{Hazan}) and statistics \cite{Ber} which are based on the alleged hardness of the the problem of distinguishing
between a random graph distributed as $G(n,1/2)$ to a random graph
with a planted clique of size $k(n)$ (with $(2+\delta)\log n<k(n)\ll \sqrt{n}$).

As another example, consider the following setting introduced by \cite{Aries} (see also \cite{ADD}): one is given a realization of a $n$-dimensional Gaussian vector $\bx:=(\bx_1,..,\bx_n)$ with i.i.d. entries.
The goal is to distinguish between the following two hypotheses.
Under the first hypothesis,  all entries in $\bx$ are i.i.d. standard normals.
Under the second hypothesis,
one is given
a family of subsets $C:=\{S_1,...,S_m\}$
such that for every $1\leq k \leq m, S_k \subseteq \{1,...,n\}$ and there
exists an $i\in \{1,\ldots,m\}$ such that, for any $\alpha\in S_i$,
$\bx_\alpha$ is a Gaussian random variable with mean $\mu>0$
and unit variance whereas for every $\alpha
\notin S_i$, $\bx_\alpha$ is standard normal.
(The second hypothesis does not specify the index $i$, only its existence).
The main question is how large $\mu$ must be such that one can
reliably distinguish between these two hypotheses.
In \cite{Aries}, one considers two situations. In
the first, $\alpha$ are vertices
in a two dimensional grid of side length $n$ and
the family $C$ is the set of all directed
(i.e., with north or east steps only)  paths of length
$l$, starting from the bottom-left corner. In the second situation
treated in \cite{Aries}, the $\alpha$s correspond to the vertices of a binary tree, and again the family $C$ consists of loopless paths starting at the root.
In \cite{Aries}, both the min-max and Bayesian (with uniform
choice of $i$ in $C$) setups are considered.
 Other choices of $C$ are considered in \cite{ADD}:
the family of all subsets of size $k$, the set of all perfect matching
in a given graph and other examples. These detection problems
have practical applications-see \cite{Aries} for details.

The Gaussian hidden clique problem is related to various applications
in statistics and computational biology \cite{Balla,Kollar}. That
detection is statistically possible when $L \gg \log n$ was
established in \cite{ADD} (the authors
consider the case where all diagonal elements are zero, but since
the detection algorithm can simply ignore the diagonal elements,
their results apply to our setting as well). Similar results for the asymmetric case were obtained by \cite{Ma}. In terms of \emph{polynomial time} detection, \cite{Desh} show that detection is possible when $L=\Theta(\sqrt{n})$ for the symmetric cases. As noted, no polynomial time algorithm is known for the Gaussian hidden clique problem when $k=o(\sqrt{n})$. In \cite{ADD} it was hypothesized that the Gaussian hidden clique problem should be difficult when $L \ll \sqrt{n}$. More specifically \cite[Pg. 16, second paragraph]{ADD} comment that ``it seems likely
that designing an efficient test in the normal setting will prove as
difficult as it proved for planted cliques". Supporting evidence for this assertion was provided in \cite{Ma}, who proved that distinguishing between a planted model similar to the Gaussian planted clique model studied in our work and the random case (with entries being independent standard Gaussians) is at least as hard as distinguishing between a graph containing a planted clique and a graph distributed the random graph $G(n,1/2)$.

There is a large body of work in RMT that addresses the effect of
low rank perturbations on various properties the spectrum
and eigenvectors of Wigner matrices
(e.g., \cite{BBAP,Furedi,Knowles}), mostly in studying the
almost sure limits and
limit distributions of extremal eigenvalues and eigenvectors.

The closest results to ours are the ones of
\cite{onatski2013asymptotic}.
In the language of the present paper, these authors consider a
rectangular matrix of the form $\bX = \lambda\, \bv_1\bv_2^{\sT}
+\bZ\in\reals^{n_1\times n_2}$ whereby $\bZ$ has i.i.d. entries 
$\bZ_{ij}\sim\normal(0,1/n_1)$, $\bv_1$ is deterministic of unit norm, and $\bv_2$ has entries which are i.i.d. $\normal(0,1/n_1)$, independent of $\bZ$. They consider the problem of testing
this distribution against $\lambda=0$. 
Setting $c=\lim_{n\to\infty}\frac{n_1}{n_2}$, it is proved in
\cite{onatski2013asymptotic} that the distribution of 
the singular values of $\bX$ under the null and the alternative are mutually contiguous if $\beta< \sqrt{c}$
and not mutually contiguous if $\beta>\sqrt{c}$. This (almost) correspond to 
the asymmetric model of Section \ref{sec-asym}, for the matrix case
$k=2$.
While \cite{onatski2013asymptotic}  derive some more refined results,
their proofs rely on advanced tools from random matrix theory
\cite{Gui}, while our proof is simpler, and generalizable to other
settings (e.g. tensors).

\section{Conclusion}
In this work we considered detection problems for GOE matrices
perturbed by a deterministic rank $1$ matrix,
including the Gaussian hidden clique problem. We have
established that spectral methods stop being effective when the norm of the perturbation drops below a threshold, which translates in the Gaussian
hidden clique problem to the size of the planted submatrix being
smaller than $(1-\epsilon)\sqrt{n}$. In identifying this threshold we have also addressed detectability issues in rank-one perturbations of matrices and tensors which might be of independent interest.

There are several open problems that arise from the current work. First, in the context of the Gaussian hidden clique problem,
it would be interesting to provide an efficient algorithm for
finding a planted submatrix when $L=o(\sqrt{n})$ or rule
out certain algorithmic (non spectral)
approaches for this problem. One direction  is to study optimization
methods such as semidefinite programming and various types of
hierarchies (e.g., Lasserre, Sherali-Adams) in dealing
with the hidden clique problem for $L=o(\sqrt{n})$.

A natural question is whether one can use the spectrum in
order to distinguish between a graph distributed as $G(n,1/2)$ and
a random graph with a planted clique of size $L < (1-\epsilon)\sqrt{n}$.
Proving that one can or cannot distinguish between these
cases  using eigenvalues is a challenging open problem.

Finally, it might be interesting to study the limitations of
spectral techniques for other problems such as coloring \cite{Kahale}
and satisfiability \cite{Flaxman}.

\section*{Acknowledgments}
We thank Iain Johnstone for bringing
\cite{onatski2013asymptotic} to our attention.

\end{document}